\documentclass[12pt,a4paper]{article}
\usepackage{latexsym, amsmath, amsfonts, amscd, amssymb, verbatim}
\usepackage{hyperref}

\def\N{I\!\!N}
\def\R{I\!\!R}

\def\mJ{\mathcal{J}}

\def\op{\bar{p}}

\def\ou{\bar{u}}\def\wu{\widetilde{u}}\def\hu{\widehat{u}}

\def\wy{\widetilde{y}}\def\hy{\widehat{y}}

\def\oB{\bar{B}}

\def\conv{{\rm conv}}

\def\st{|\;}\def\bst{\big|\;}

\let\saveae\ae
\renewcommand{\ae}{\text{\rm\saveae}}

\begin{document}

\newtheorem{Theorem}{Theorem}[section]
\newtheorem{Proposition}{Proposition}[section]
\newtheorem{Remark}{Remark}[section]
\newtheorem{Lemma}{Lemma}[section]
\newtheorem{Corollary}{Corollary}[section]
\newtheorem{Definition}{Definition}[section]
\newtheorem{Example}{Example}[section]
\newtheorem{CounterExample}{Counter-Example}[section]
\renewcommand{\theequation}{\thesection.\arabic{equation}}
\normalsize

\title{\bf Stability for Bang-Bang Control Problems\\ of Partial Differential Equations\footnote{This research is
supported by the Alexander von Humboldt Foundation.
The second author was partially supported by DFG under grant number Wa 3626/1-1.
}}

\author{Nguyen Thanh Qui\footnote{College of Information and Communication Technology,
        Can Tho University, Campus II, 3/2 Street, Can Tho, Vietnam;
        ntqui@cit.ctu.edu.vn. Current address (01.07.2016-30.06.2018): Institut f\"{u}r Mathematik,
        Universit\"{a}t W\"{u}rzburg, Emil-Fischer-Str.~30, 97074 W\"{u}rzburg, Germany;
        thanhqui.nguyen@mathematik.uni-wuerzburg.de.}\quad and\quad
        Daniel Wachsmuth\footnote{Institut f\"{u}r Mathematik, Universit\"{a}t W\"{u}rzburg,
        Emil-Fischer-Str.~30, 97074 W\"{u}rzburg, Germany;
        daniel.wachsmuth@mathematik.uni-wuerzburg.de.}}

\maketitle
\date{}

\noindent {\bf Abstract.} In this paper, we investigate solution
stability for control problems of partial differential equations
with the cost functional not involving the usual quadratic term for
the control. We first establish a sufficient optimality condition
for the optimal control problems with bang-bang controls. Then we
obtain criteria for solution stability for the optimal control
problems of bang-bang controls under linear perturbations. We prove
H\"older stability of optimal controls in $L^1$.



\section{Introduction}

The aim of this paper is to study solution stability for bang-bang
optimal control problems of elliptic partial differential equations
(PDEs) under linear perturbations. In particular, we are interested
in the optimal control problems where the cost functional does not
involve the control in an explicit form as follows
\begin{equation}\label{OptConPro}
\begin{cases}
    {\rm Minimize}\quad J(u)=\displaystyle\int_\Omega L\big(x,y_u(x)\big)dx\\
    {\rm subject\ to}\quad\alpha(x)\leq u(x)\leq\beta(x)\quad\mbox{for a.e.}\ x\in\Omega,
\end{cases}
\end{equation}
where $y_u$ is the weak solution of the Dirichlet problem
\begin{equation}\label{StateEq}
\begin{cases}
\begin{aligned}
    Ay+f(x,y)&=u\ &&\mbox{in}\ \Omega\\
            y&=0  &&\mbox{on}\ \Gamma.
\end{aligned}
\end{cases}
\end{equation}
In general, local solutions $\bar u$ of this problem have the so-called bang-bang property: for
almost all $x\in \Omega$ it holds $\bar u(x)\in \{\alpha(x),\beta(x)\}$.

Motivated by the second-order sufficient optimality conditions for
bang-bang optimal control problem obtained in \cite{Cas12SICON},
\cite{CasDWchGWch17} and the results on numerical methods obtained
in \cite{PonWch16OPTIM}, \cite{PonWch17}, we
investigate the perturbed optimal control problem
\begin{equation}
\begin{cases}
    {\rm Minimize}\quad \mJ(u,e)=J(u+e_y)+(e_J,y_{u+e_y})_{L^2(\Omega)}\\
    {\rm subject\ to}\quad u\in\mathcal{U}_{ad},
\end{cases}
\end{equation}
where $y_{u+e_y}$ is the weak solution of the
perturbed Dirichlet problem
\begin{equation}
\begin{cases}
\begin{aligned}
    Ay+f(x,y)&=u+e_y\ &&\mbox{in}\ \Omega\\
            y&=0      &&\mbox{on}\ \Gamma.
\end{aligned}
\end{cases}
\end{equation}
Here, $e=(e_y,e_J)$ is a given perturbation. We will show that the
perturbed problem has local solutions near solutions of the original
problem that satisfy a second-order condition, see
Theorem~\ref{ThmHldrEstm}. Under an additional assumption on the
bang-bang control, we prove local H\"older stability in
$L^1(\Omega)$ of optimal controls with respect to the perturbations,
see Theorem~\ref{ThmStabKKT}.

Stability results for optimal control problems can be found, for
instance, in \cite{MalaTrol00CC,MorNgia14SIOPT}. However, these
results are not applicable in our situation, as they require second-order
growth of the cost functional in $L^2(\Omega)$, which is not
fulfilled in bang-bang control problems. Instead, our analysis
relies on a second-order condition due to Casas \cite{Cas12SICON},
which substantially weakens the second-order condition while still
implying local stability, see, e.g., Theorem~\ref{ThmSSC} below.

In addition, we use an assumption that controls the growth of the adjoint state near jumps of the control.
This condition was used recently in  \cite{PonWch17} to obtain regularization error estimates of Tikhonov regularization
of the bang-bang problem.

Quite a number of stability results are available for optimal
control problems with bang-bang controls subject to ordinary differential equations, see,
e.g., \cite{Fel03SICON,Fel09CC}. The stability is based on
assumptions on the switching function, which imply our condition
\textbf{(A4.$\ae$)}. In addition, second-order conditions on
switching times are imposed. While the methods of proof are not
directly transferable, stability of controls with respect to $L^1$-norms is obtained, as
in our case.

\section{Preliminaries}

Let us assume that $\Omega\subset\R^N$ with $N\in\{1,2,3\}$ and
$\alpha,\beta\in L^\infty(\Omega)$ with $\alpha(x)\leq\beta(x)$ for
a.e. $x\in\Omega$. Moreover, the functions $L,f:\Omega\times\R\to\R$
are Carath\'eodory functions of class $\mathcal{C}^2$ with respect
to the second variable satisfying the following assumptions.

\textbf{(A1)} The function $f(\cdot,0)\in L^{\op}(\Omega)$ with
$\op>N/2$,
$$\dfrac{\partial f}{\partial y}(x,y)\geq0\quad\mbox{for a.e.}\ x\in\Omega,$$
and for all $M>0$ there exists a constant $C_{f,M}>0$ such that
$$\left|\dfrac{\partial f}{\partial y}(x,y)\right|+\left|\dfrac{\partial^2f}{\partial y^2}(x,y)\right|\leq C_{f,M}
  \quad\mbox{for a.e.}\ x\in\Omega\ \mbox{and}\ |y|\leq M.$$
For every $M>0$ and $\varepsilon>0$ there exists $\delta>0$,
depending on $M$ and $\varepsilon$ such that
$$\left|\dfrac{\partial^2f}{\partial y^2}(x,y_2)-\dfrac{\partial^2f}{\partial y^2}(x,y_1)\right|<\varepsilon
  \quad\mbox{if}\ |y_1|,|y_2|\leq M,|y_2-y_1|\leq\delta,\ \mbox{and for a.e.}\ x\in\Omega.$$

\textbf{(A2)} The function $L(\cdot,0)\in L^1(\Omega)$ and for all
$M>0$ there are a constant $C_{L,M}>0$ and a function $\psi_M\in
L^{\op}(\Omega)$ such that for every $|y|\leq M$ and almost all
$x\in\Omega$,
$$\left|\dfrac{\partial L}{\partial y}(x,y)\right|\leq\psi_M(x),\quad
  \left|\dfrac{\partial^2L}{\partial y^2}(x,y)\right|\leq C_{L,M}.$$
For every $M>0$ and $\varepsilon>0$ there exists $\delta>0$,
depending on $M$ and $\varepsilon$ such that
$$\left|\dfrac{\partial^2L}{\partial y^2}(x,y_2)-\dfrac{\partial^2L}{\partial y^2}(x,y_1)\right|<\varepsilon
  \quad\mbox{if}\ |y_1|,|y_2|\leq M,|y_2-y_1|\leq\delta,\ \mbox{and for a.e.}\ x\in\Omega.$$

\textbf{(A3)} $\Omega$ is an open and bounded domain in $\R^N$ with
Lipschitz boundary $\Gamma$, and $A$ denotes a second-order
differential elliptic operator of the form
$$Ay(x)=-\sum_{i,j=1}^N\partial_{x_j}\big(a_{ij}(x)\partial_{x_i}y(x)\big);$$
the coefficients $a_{ij}\in C(\bar\Omega)$ satisfy
$$\lambda_A|\xi|^2\leq\sum_{i,j=1}^Na_{ij}(x)\xi_i\xi_j,\ \forall\xi\in\R^N,\ \mbox{for a.e.}\ x\in\Omega,$$
for some $\lambda_A>0$.

We denote
$$\mathcal{U}_{ad}=\big\{u\in L^\infty(\Omega)\bst\alpha(x)\leq u(x)\leq\beta(x)\ \mbox{for a.e.}\ x\in\Omega\big\}.$$
Observe that $\mathcal{U}_{ad}$ is nonempty, closed, bounded, and
convex in $L^p(\Omega)$ whenever $1\leq p\leq\infty$.

We refer the reader to \cite[Chapter~4]{Trolt10B} for the proofs of
the following results on the solution of the state equation
\eqref{StateEq}. For every $u\in L^p(\Omega)$ with $p>N/2$, equation
\eqref{StateEq} has a unique weak solution $y_u\in H^1_0(\Omega)\cap
C(\bar\Omega)$. In addition, there exists a constant
$M_{\alpha,\beta}$ such that
\begin{equation}\label{EstSolEqSt}
    \|y_u\|_{H^1_0(\Omega)}+\|y_u\|_{C(\bar\Omega)}\leq M_{\alpha,\beta},\ \forall u\in\mathcal{U}_{ad}.
\end{equation}
The control-to-state mapping $G:L^2(\Omega)\to H^1_0(\Omega)\cap
C(\bar\Omega)$ defined by $G(u)=y_u$ is of class $\mathcal{C}^2$.
Moreover, for every $v\in L^2(\Omega)$, $z_{u,v}=G'(u)v$ is the
unique weak solution of
\begin{equation}\label{EqSolZuv}
\begin{cases}
  \begin{aligned}
     Az+\frac{\partial f}{\partial y}(x,y)z&=v\ &&\mbox{in}\ \Omega\\
                                          z&=0  &&\mbox{on}\ \Gamma,
  \end{aligned}
\end{cases}
\end{equation}
and for any $v_1,v_2\in L^2(\Omega)$, $w_{v_1,v_2}=G''(u)(v_1,v_2)$
is the unique weak solution of
\begin{equation}\label{EqSolSeGvv}
\begin{cases}
  \begin{aligned}
     Aw+\frac{\partial f}{\partial y}(x,y)w+\frac{\partial^2f}{\partial y^2}(x,y)z_{u,v_1}z_{u,v_2}
      &=0\ &&\mbox{in}\ \Omega\\
     w&=0  &&\mbox{on}\ \Gamma,
  \end{aligned}
\end{cases}
\end{equation}
where $y=G(u)$ and $z_{u,v_i}=G'(u)v_i$ for $i=1,2$.

By assumption \textbf{(A2)}, using the latter results and applying
the chain rule we deduce that the cost functional
$J:L^2(\Omega)\to\R$ is of class $\mathcal{C}^2$, and the first and
second derivatives of $J(\cdot)$ are given by
\begin{equation}\label{FDeriCost}
    J'(u)v=\int_\Omega\varphi_u(x)v(x)dx,
\end{equation}
and
\begin{equation}\label{SDeriCost}
    J''(u)(v_1,v_2)=\int_\Omega\left(\dfrac{\partial^2L}{\partial y^2}\big(x,y_u(x)\big)-\varphi_u(x)
    \dfrac{\partial^2f}{\partial y^2}\big(x,y_u(x)\big)\right)z_{u,v_1}(x)z_{u,v_2}(x)dx,
\end{equation}
where $z_{u,v_i}=G'(u)v_i$ for $i=1,2$, and $\varphi_u\in
H^1_0(\Omega)\cap C(\bar\Omega)$ is the adjoint state of $y_u$
defined as the unique  weak solution of
\begin{equation}\label{AdjStaEq}
\begin{cases}
\begin{aligned}
    A^*\varphi+\dfrac{\partial f}{\partial y}(x,y_u)\varphi
                     &=\dfrac{\partial L}{\partial y}(x,y_u)\ &&\mbox{in}\ \Omega\\
             \varphi &=0                                      &&\mbox{on}\ \Gamma,
\end{aligned}
\end{cases}
\end{equation}
where $A^*$ is the adjoint operator of $A$.

For any $p\in[1,\infty]$, we denote $\oB^p_\varepsilon(\ou)$ the
closed ball in the space $L^p(\Omega)$ with the center at $\ou\in
L^p(\Omega)$ and the radius $\varepsilon>0$, i.e.,
$$\oB^p_\varepsilon(\ou)=\{v\in L^p(\Omega)\bst\|v-\ou\|_{L^p(\Omega)}\leq\varepsilon\}.$$
An element $\ou\in\mathcal{U}_{ad}$ is said to be a
\emph{solution}/\emph{global minimum} of problem~\eqref{OptConPro}
if $J(\ou)\leq J(u)$ for all $u\in\mathcal{U}_{ad}$. We will say
that $\ou$ is a \emph{local solution}/\emph{local minimum} of
problem~\eqref{OptConPro} in the sense of $L^p(\Omega)$ if there
exists a closed ball $\oB^p_\varepsilon(\ou)$ such that $J(\ou)\leq
J(u)$ for all $u\in\mathcal{U}_{ad}\cap\oB^p_\varepsilon(\ou)$. The
local solution $\ou$ is called \emph{strict} if $J(\ou)<J(u)$ holds
for all $u\in\mathcal{U}_{ad}\cap\oB^p_\varepsilon(\ou)$ with
$u\neq\ou$.

Under the above assumptions, solutions of problem~\eqref{OptConPro}
exist. We introduce the space $Y=H^1_0(\Omega)\cap C(\bar\Omega)$
endowed with the norm
$$\|y\|_Y=\|y\|_{H^1_0(\Omega)}+\|y\|_{L^\infty(\Omega)}.$$
We know that, see e.g. \cite[Chapter~4]{Trolt10B}, if $\ou$ is a
local solution of problem~\eqref{OptConPro} in the sense of
$L^p(\Omega)$, then there exist a unique state $y_{\ou}\in Y$ and a
unique adjoint state $\varphi_{\ou}\in Y$ satisfying the first-order
optimality system
\begin{equation}\label{StateEqSol}
\begin{cases}
\begin{aligned}
    Ay_{\ou}+f(x,y_{\ou})&=\ou\ &&\mbox{in}\ \Omega\\
                  y_{\ou}&=0    &&\mbox{on}\ \Gamma,
\end{aligned}
\end{cases}
\end{equation}
\begin{equation}\label{AdjEq}
\begin{cases}
\begin{aligned}
    A^*\varphi_{\ou}+\dfrac{\partial f}{\partial y}(x,y_{\ou})\varphi_{\ou}
                 &=\dfrac{\partial L}{\partial y}(x,y_{\ou})\ &&\mbox{in}\ \Omega\\
    \varphi_{\ou}&=0                                          &&\mbox{on}\ \Gamma,
\end{aligned}
\end{cases}
\end{equation}
\begin{equation}\label{VarIneq}
    \int_\Omega\varphi_{\ou}(x)\big(u(x)-\ou(x)\big)dx\geq0,\ \forall u\in\mathcal{U}_{ad}.
\end{equation}

\section{Optimality conditions for bang-bang controls}

Let $\ou$ be locally optimal for problem~\eqref{OptConPro} in the
sense of $L^p(\Omega)$ with $p\in[1,\infty]$. From \eqref{VarIneq},
we deduce that
\begin{equation}\label{ValOuVarp}
\ou(x)=\begin{cases}
            \alpha(x),  &\mbox{if}\ \varphi_{\ou}(x)>0\\
            \beta(x),   &\mbox{if}\ \varphi_{\ou}(x)<0
\end{cases}
\quad\mbox{and}\quad
\varphi_{\ou}(x)\begin{cases}
            \geq0,   &\mbox{if}\ \ou(x)=\alpha(x)\\
            \leq0,   &\mbox{if}\ \ou(x)=\beta(x)\\
            =0,      &\mbox{if}\ \alpha(x)<\ou(x)<\beta(x).
\end{cases}
\end{equation}
Let us consider the case where the set
$\{x\in\Omega\st\varphi_{\ou}(x)=0\}$ has a zero Lebesgue measure.
Then, it follows from \eqref{ValOuVarp} that
$\ou(x)\in\{\alpha(x),\beta(x)\}$ for a.e. $x\in\Omega$, i.e., $\ou$
is a bang-bang control.

The goal of this section is to provide sufficient optimality
conditions for local optimality of a bang-bang control
$\ou\in\mathcal{U}_{ad}$ satisfying the first-order optimality
system \eqref{StateEqSol}-\eqref{VarIneq}. The sufficient optimality
conditions are established via the second-order derivative of the
cost functional $J(\cdot)$. For this reason, a cone of critical
directions is given. Let us first consider the natural cone of
critical directions associated with $\ou$ defined by
\begin{equation}\label{CriCone}
C_{\ou}=\left\{v\in L^2(\Omega)\Bigg\st v(x)
\begin{cases}
    \geq0\quad\mbox{if}\ \ou(x)=\alpha(x)\\
    \leq0\quad\mbox{if}\ \ou(x)=\beta(x)\\
       =0\quad\mbox{if}\ \varphi_{\ou}(x)\neq0
\end{cases}\right\}.
\end{equation}
Then, the second-order necessary conditions for local optimality can
be written in the form
\begin{equation}\label{SeOrNeCnd}
    J''(\ou)v^2\geq0,\ \forall v\in C_{\ou},
\end{equation}
see, e.g., \cite[Section~6.3]{BonSha00B}. However, it follows from
\eqref{ValOuVarp} and \eqref{CriCone} that if $\ou$ is a bang-bang
control, then $C_{\ou}=\{0\}$. Therefore,
condition~\eqref{SeOrNeCnd} is trivial and it does not provide any
information. To overcome this drawback, following \cite{Cas12SICON}
we increase the cone $C_{\ou}$ to an extended cone given as follows.
For every $\tau\geq0$, we define
\begin{equation}\label{ExCriConce}
C^\tau_{\ou}=\left\{v\in L^2(\Omega)\Bigg\st v(x)
\begin{cases}
    \geq0\quad\mbox{if}\ \ou(x)=\alpha(x)\\
    \leq0\quad\mbox{if}\ \ou(x)=\beta(x)\\
       =0\quad\mbox{if}\ |\varphi_{\ou}(x)|>\tau
\end{cases}\right\}.
\end{equation}
It is clear that $C^0_{\ou}=C_{\ou}$, and $C^\tau_{\ou}$ is bigger
than $C_{\ou}$ in general for $\tau>0$.

In this paper, in order to consider bang-bang controls $\ou$ of
problem~\eqref{OptConPro} we are interested in the case where the
set $\{x\in\Omega\st\varphi_{\ou}(x)=0\}$ has a zero Lebesgue
measure. As a consequence, the following assumption posed on the
adjoint state $\varphi_{\ou}$ is natural; see \cite{CasDWchGWch17}.

\textbf{(A4)} Assume that $\ou\in\mathcal{U}_{ad}$ and it satisfies
the first-order optimality system
\eqref{StateEqSol}-\eqref{VarIneq}, and the condition below
\begin{equation}\label{AsAdVphi}
    \exists K>0\ \mbox{such that}\ \big|\{x\in\Omega:|\varphi_{\ou}(x)|\leq\varepsilon\}\big|\leq
    K\varepsilon,\ \forall\varepsilon>0.
\end{equation}
In \eqref{AsAdVphi}, we denote $|\cdot|$ the Lebesgue measure.

\begin{Proposition}{\rm(See \cite[Proposition~2.7]{CasDWchGWch17})}\label{PropFCdCWW}
Assume that {\rm\textbf{(A1)}-\textbf{(A4)}} hold. Then, there
exists $\kappa>0$ such that
\begin{equation}
    J'(\ou)(u-\ou)\geq\kappa\|u-\ou\|^2_{L^1(\Omega)},\ \forall u\in\mathcal{U}_{ad}.
\end{equation}
\end{Proposition}

We are going to extend the result given in
Proposition~\ref{PropFCdCWW}, where \textbf{(A4)} is replace with
assumption \textbf{(A4.$\ae$)} stated below.

\textbf{(A4.$\ae$)} Assume that $\ou\in\mathcal{U}_{ad}$ and it
satisfies the first-order optimality system
\eqref{StateEqSol}-\eqref{VarIneq}, and the following condition
\begin{equation}\label{AsmAdVrphi}
    \exists K>0,\exists\ae>0\ \mbox{such that}\ \big|\{x\in\Omega:|\varphi_{\ou}(x)|\leq\varepsilon\}\big|
    \leq K\varepsilon^{\ae},\ \forall\varepsilon>0.
\end{equation}

\begin{Proposition}\label{PropFstCd}
Assume that {\rm\textbf{(A1)}-\textbf{(A3)}} and
{\rm\textbf{(A4.$\ae$)}} hold. Then, there exists $\kappa>0$ such
that
\begin{equation}\label{FOrdL1}
    J'(\ou)(u-\ou)\geq\kappa\|u-\ou\|^{1+\frac{1}{\ae}}_{L^1(\Omega)},\ \forall u\in\mathcal{U}_{ad}.
\end{equation}
\end{Proposition}
\begin{proof}
Given $u\in\mathcal{U}_{ad}$, we put
$$\varepsilon:=\Big((2\|\beta-\alpha\|_{L^\infty(\Omega)}K)^{-1}\|u-\ou\|_{L^1(\Omega)}\Big)^{1/\ae}
  \quad\mbox{and}\quad E_\varepsilon:=\{x\in\Omega:|\varphi_{\ou}|\geq\varepsilon^{\ae}\}.$$
Then, we have
$$\begin{aligned}
    J'(\ou)(u-\ou)&=\int_\Omega|\varphi_{\ou}||u-\ou|dx\geq\int_{E_\varepsilon}|\varphi_{\ou}||u-\ou|dx\\
                  &\geq\varepsilon\|u-\ou\|_{L^1(E_\varepsilon)}
                   =\varepsilon\big(\|u-\ou\|_{L^1(\Omega)}-\|u-\ou\|_{L^1(\Omega\setminus E_\varepsilon)}\big).
\end{aligned}$$
Since $|\Omega\setminus E_\varepsilon|\leq K\varepsilon^{\ae}$ by
\eqref{AsmAdVrphi}, we obtain
$$\|u-\ou\|_{L^1(\Omega\setminus E_\varepsilon)}\leq\|\beta-\alpha\|_{L^\infty(\Omega)}K\varepsilon^{\ae}.$$
Consequently, we deduce that
$$\begin{aligned}
    J'(\ou)(u-\ou)&\geq\varepsilon\big(\|u-\ou\|_{L^1(\Omega)}-\|u-\ou\|_{L^1(\Omega\setminus E_\varepsilon)}\big)\\
                  &\geq\varepsilon\big(\|u-\ou\|_{L^1(\Omega)}-\|\beta-\alpha\|_{L^\infty(\Omega)}
                   K\varepsilon^{\ae}\big)\\
                  &\geq\frac{\varepsilon}{2}\|u-\ou\|_{L^1(\Omega)}=\kappa\|u-\ou\|^{1+\frac{1}{\ae}}_{L^1(\Omega)},
\end{aligned}$$
where
$\kappa=2^{-1}\big(2\|\beta-\alpha\|_{L^\infty(\Omega)}K\big)^{\frac{1}{\ae}}.$
$\hfill\Box$
\end{proof}

\medskip
The following theorem provides us with a second-order sufficient
optimality condition for the bang-bang control
problem~\eqref{OptConPro}.

\begin{Theorem}\label{ThmSSC}
Assume that $\ou$ is a feasible control for problem
\eqref{OptConPro} satisfying {\rm\textbf{(A1)}-\textbf{(A3)}} and
{\rm\textbf{(A4.$\ae$)}} and assume that there exist $\delta>0$ and
$\tau>0$ such that
\begin{equation}\label{SOrdCd}
    J''(\ou)v^2\geq\delta\|z_v\|^2_{L^2(\Omega)},\ \forall v\in C^\tau_{\ou},
\end{equation}
where $z_v=G'(\ou)v$ is the solution of \eqref{EqSolZuv} for
$y=y_{\ou}$. Then, there exists $\varepsilon>0$ such that
\begin{equation}\label{GrOpCd}
    J(\ou)+\frac{\kappa}{2}\|u-\ou\|^{1+\frac{1}{\ae}}_{L^1(\Omega)}+\frac{\delta}{8}\|z_{u-\ou}\|^2_{L^2(\Omega)}
    \leq J(u),\ \forall u\in\mathcal{U}_{ad}\cap\oB^2_\varepsilon(\ou),
\end{equation}
with $z_{u-\ou}=G'(\ou)(u-\ou)$ and $\kappa$ being given in
Proposition~\ref{PropFstCd}.
\end{Theorem}

In order to prove Theorem~\ref{ThmSSC} we need the following
technical lemmas.

\begin{Lemma}\label{LemCas26}{\rm(See \cite[Lemma~2.6]{Cas12SICON})}
For any $u\in\mathcal{U}_{ad}$ and $v\in L^2(\Omega)$, denote
$z_{u,v}=G'(u)v$. Also, we set $z_v=G'(\ou)v$. Then there exist
constants $C_2>0$ and $C_3>0$ such that
$$\|z_{u,v}-z_v\|_Y\leq C_2\|u-\ou\|_{L^2(\Omega)}\|z_v\|_{L^2(\Omega)},\ \forall v\in L^2(\Omega),$$
$$\|z_{u,v}\|_{L^2(\Omega)}\leq C_3\|v\|_{L^1(\Omega)},\ \forall v\in L^1(\Omega),$$
for all $u\in\mathcal{U}_{ad}$.
\end{Lemma}

\begin{Lemma}\label{LemCas27}{\rm(See \cite[Lemma~2.7]{Cas12SICON})}
For every $\varepsilon>0$, there exists $\rho>0$ such that for
$u\in\mathcal{U}_{ad}$ with $\|u-\ou\|_{L^2(\Omega)}\leq\rho$ the
following inequality holds
$$\big|J''(u)v^2-J''(\ou)v^2\big|\leq\varepsilon\|z_v\|^2_{L^2(\Omega)},\ \forall v\in L^2(\Omega).$$
\end{Lemma}

We are going to prove Theorem~\ref{ThmSSC}, several techniques used
in this proof are similar as in the proof of
\cite[Theorem~2.4]{Cas12SICON}.

\textbf{\emph{Proof of Theorem~\ref{ThmSSC}}.} Let us define the
function $F:\Omega\times\mathcal{U}_{ad}\to L^\infty(\Omega)$ by
$$F(x,u)=\frac{\partial^2L}{\partial y^2}(x,y_u)-\varphi_u\frac{\partial^2f}{\partial y^2}(x,y_u).$$
Then $F$ is well-defined due to the assumptions on $f$ and $L$, and
$$\|y_u\|_{L^\infty(\Omega)}+\|\varphi_u\|_{L^\infty(\Omega)}\leq M,\ \forall u\in\mathcal{U}_{ad},$$
for some $M>0$. We also have
$$\|F(x,u)\|_\infty\leq K_M,\ \forall u\in\mathcal{U}_{ad}.$$
We observe that for every $v,w\in L^2(\Omega)$ and
$u\in\mathcal{U}_{ad}$,
\begin{equation}
    |J''(u)(v,w)|=\left|\int_\Omega F(x,u(x))z_{u,v}z_{u,w}dx\right|\leq
    K_M\|z_{u,v}\|_{L^2(\Omega)}\|z_{u,w}\|_{L^2(\Omega)}.
\end{equation}
By Lemma~\ref{LemCas27}, there exists $\varepsilon_0>0$ such that
$$|J''(u)v^2-J''(\ou)v^2|\leq\frac{\delta}{4}\|z_v\|^2_{L^2(\Omega)}\quad
  \mbox{when}\ \|u-\ou\|_{L^2(\Omega)}\leq\varepsilon_0.$$
Take $\varepsilon$ with $0<\varepsilon<\varepsilon_0$ such that
\begin{equation}\label{ChseVarep}
    \frac{\tau}{2C^2_3\varepsilon\sqrt{|\Omega|}}-\frac{K_M(C_2\varepsilon+1)^2}{2}
    -\frac{2K^2_M(C_2\varepsilon+1)^4}{\delta}\geq\frac{\delta}{4},
\end{equation}
where $C_2$ and $C_3$ are given in Lemma~\ref{LemCas26}.

Given $u\in\oB^2_\varepsilon(\ou)\cap\mathcal{U}_{ad}$, we define
$$v(x)=\begin{cases}
          u(x)-\ou(x),   &\mbox{if}\ |\varphi_{\ou}(x)|\leq\tau\\
          0,             &\mbox{otherwise}
       \end{cases}
  \qquad\mbox{and}\quad w=(u-\ou)-v.$$
We see that $v\in C^\tau_{\ou}$. Making a Taylor expansion of second
order we obtain
$$J(u)=J(\ou)+J'(\ou)(u-\ou)+\frac{1}{2}J''(\hu)(u-\ou)^2$$
for some $\hu=\ou+\theta(u-\ou)$ with $\theta\in(0,1)$. From
\eqref{ValOuVarp} and $u-\ou=v+w$ we have
$$\begin{aligned}
    J(u)
    &=J(\ou)+\frac{1}{2}J'(\ou)(u-\ou)+\frac{1}{2}\int_\Omega|\varphi_{\ou}||u-\ou|dx+\frac{1}{2}J''(\hu)(u-\ou)^2\\
    &=J(\ou)+\frac{1}{2}J'(\ou)(u-\ou)+\frac{1}{2}\int_\Omega|\varphi_{\ou}||u-\ou|dx\\
    &\qquad\quad\ +\frac{1}{2}J''(\ou)v^2+\frac{1}{2}\big(J''(\hu)v^2-J''(\ou)v^2\big)
     +\frac{1}{2}J''(\hu)w^2+J''(\hu)(v,w)\\
    &\mbox{(with Proposition~\ref{PropFstCd})}\\
    &\geq J(\ou)+\frac{1}{2}\kappa\|u-\ou\|^{1+\frac{1}{\ae}}_{L^1(\Omega)}
     +\frac{1}{2}\int_\Omega|\varphi_{\ou}||w|dx+\frac{\delta}{2}\|z_v\|^2_{L^2(\Omega)}
     -\frac{\delta}{8}\|z_v\|^2_{L^2(\Omega)}\\
    &\qquad\quad\ -\frac{1}{2}K_M\|z_{\hu,w}\|^2_{L^2(\Omega)}
     -K_M\|z_{\hu,v}\|_{L^2(\Omega)}\|z_{\hu,w}\|_{L^2(\Omega)}\\
    &\geq J(\ou)+\frac{1}{2}\kappa\|u-\ou\|^{1+\frac{1}{\ae}}_{L^1(\Omega)}
     +\frac{1}{2}\tau\|w\|_{L^1(\Omega)}+\frac{3\delta}{8}\|z_v\|^2_{L^2(\Omega)}\\
    &\qquad\quad\ -\frac{1}{2}K_M\|z_{\hu,w}\|^2_{L^2(\Omega)}
     -K_M\|z_{\hu,v}\|_{L^2(\Omega)}\|z_{\hu,w}\|_{L^2(\Omega)}.
\end{aligned}$$

Note that
$\|\hu-\ou\|_{L^2(\Omega)}=\|\theta(u-\ou)\|_{L^2(\Omega)}\leq\|u-\ou\|_{L^2(\Omega)}\leq\varepsilon$.
Hence, by Lemma~\ref{LemCas26} we deduce that
$$\begin{aligned}
    \|z_{\hu,w}\|_{L^2(\Omega)}
    &\leq\|z_{\hu,w}-z_w\|_{L^2(\Omega)}+\|z_w\|_{L^2(\Omega)}\\
    &\leq C_2\|\hu-\ou\|_{L^2(\Omega)}\|z_w\|_{L^2(\Omega)}+\|z_w\|_{L^2(\Omega)}\\
    &\leq(C_2\varepsilon+1)\|z_w\|_{L^2(\Omega)}.
\end{aligned}$$
Similarly, we also have
$$\|z_{\hu,v}\|_{L^2(\Omega)}\leq(C_2\varepsilon+1)\|z_v\|_{L^2(\Omega)}.$$
By arguing the same as in the proof of Theorem~2.4 in
\cite{Cas12SICON} we obtain
$$\frac{1}{C^2_3\varepsilon\sqrt{|\Omega|}}\|z_w\|^2_{L^2(\Omega)}\leq\|w\|_{L^1(\Omega)}.$$
Therefore, we get
$$\begin{aligned}
    J(u)
    &\geq J(\ou)+\frac{1}{2}\kappa\|u-\ou\|^{1+\frac{1}{\ae}}_{L^1(\Omega)}
     +\frac{1}{2}\tau\|w\|_{L^1(\Omega)}+\frac{3\delta}{8}\|z_v\|^2_{L^2(\Omega)}\\
    &\qquad\quad\ -\frac{1}{2}K_M\|z_{\hu,w}\|^2_{L^2(\Omega)}
     -K_M\|z_{\hu,v}\|_{L^2(\Omega)}\|z_{\hu,w}\|_{L^2(\Omega)}\\
    &\geq J(\ou)+\frac{1}{2}\kappa\|u-\ou\|^{1+\frac{1}{\ae}}_{L^1(\Omega)}
     +\frac{\tau}{2C^2_3\varepsilon\sqrt{|\Omega|}}\|z_w\|^2_{L^2(\Omega)}+\frac{3\delta}{8}\|z_v\|^2_{L^2(\Omega)}\\
    &\qquad\quad\ -\frac{1}{2}K_M(C_2\varepsilon+1)^2\|z_w\|^2_{L^2(\Omega)}
     -K_M(C_2\varepsilon+1)^2\|z_v\|_{L^2(\Omega)}\|z_w\|_{L^2(\Omega)}\\
    &\geq J(\ou)+\frac{1}{2}\kappa\|u-\ou\|^{1+\frac{1}{\ae}}_{L^1(\Omega)}
     +\frac{3\delta}{8}\|z_v\|^2_{L^2(\Omega)}-\frac{\delta}{8}\|z_v\|^2_{L^2(\Omega)}\\
    &\qquad\quad\ +\bigg(\frac{\tau}{2C^2_3\varepsilon\sqrt{|\Omega|}}-\frac{K_M(C_2\varepsilon+1)^2}{2}
     -\frac{2K^2_M(C_2\varepsilon+1)^4}{\delta}\bigg)\|z_w\|^2_{L^2(\Omega)}.
\end{aligned}$$
From this and using \eqref{ChseVarep} we obtain
$$\begin{aligned}
    J(u)
    &\geq J(\ou)+\frac{1}{2}\kappa\|u-\ou\|^{1+\frac{1}{\ae}}_{L^1(\Omega)}
     +\frac{\delta}{4}\|z_v\|^2_{L^2(\Omega)}+\frac{\delta}{4}\|z_w\|_{L^2(\Omega)}\\
    &\geq J(\ou)+\frac{1}{2}\kappa\|u-\ou\|^{1+\frac{1}{\ae}}_{L^1(\Omega)}
     +\frac{\delta}{8}\|z_v+z_w\|^2_{L^2(\Omega)}\\
    &=J(\ou)+\frac{1}{2}\kappa\|u-\ou\|^{1+\frac{1}{\ae}}_{L^1(\Omega)}
     +\frac{\delta}{8}\|z_{u-\ou}\|_{L^2(\Omega)},
\end{aligned}$$
which yields \eqref{GrOpCd}. $\hfill\Box$

\section{Stability for bang-bang control problems}

In this section, we investigate stability of
problem~\eqref{OptConPro} for bang-bang controls, where the state
equation and the cost functional undergo linear perturbations.

We now consider the perturbed control problem of
problem~\eqref{OptConPro} as follows
\begin{equation}\label{PerProWtCd}
\begin{cases}
    {\rm Minimize}\quad \mJ(u,e)=J(u+e_y)+(e_J,y_{u+e_y})_{L^2(\Omega)}\\
    {\rm subject\ to}\quad u\in\mathcal{U}_{ad},
\end{cases}
\end{equation}
where $J(\cdot)$ is the cost functional of
problem~\eqref{OptConPro}, $y_{u+e_y}$ is the weak solution of the
perturbed Dirichlet problem
\begin{equation}\label{PerStaEqWtCd}
\begin{cases}
\begin{aligned}
    Ay+f(x,y)&=u+e_y\ &&\mbox{in}\ \Omega\\
            y&=0      &&\mbox{on}\ \Gamma,
\end{aligned}
\end{cases}
\end{equation}
and $e_J\in L^2(\Omega)$, $e_y\in L^2(\Omega)$ are parameters.

In what follows, let us put $e=(e_J,e_y)\in L^2(\Omega)\times
L^2(\Omega)$ and impose the sum norm in the product space
$E:=L^2(\Omega)\times L^2(\Omega)$, i.e.,
\begin{equation}\label{DefNormE}
    \|e\|_E=\|e_J\|_{L^2(\Omega)}+\|e_y\|_{L^2(\Omega)}.
\end{equation}

\subsection{Existence and stability of solutions of perturbed problems}

In this subsection, we are interested in the existence and stability
of solutions of the perturbed control problem~\eqref{PerProWtCd}.
Given an arbitrary $\varepsilon>0$, we first consider the following
auxiliary perturbed control problem
\begin{equation}\label{PerturbPro}
\begin{cases}
    {\rm Minimize}\quad \mJ(u,e)=J(u+e_y)+(e_J,y_{u+e_y})_{L^2(\Omega)}\\
    {\rm subject\ to}\quad u\in\mathcal{U}^\varepsilon_{ad}:=\mathcal{U}_{ad}\cap\oB^2_\varepsilon(\ou),
\end{cases}
\end{equation}
where $y_{u+e_y}=G(u+e_y)$ is the weak solution of the perturbed
Dirichlet problem
\begin{equation}\label{PerStateEq}
\begin{cases}
\begin{aligned}
    Ay+f(x,y)&=u+e_y\ &&\mbox{in}\ \Omega\\
            y&=0      &&\mbox{on}\ \Gamma,
\end{aligned}
\end{cases}
\end{equation}
and $e_J\in L^2(\Omega)$, $e_y\in L^2(\Omega)$ are parameters.

\begin{Theorem}\label{ThmAxExSol}
Assume that {\rm\textbf{(A1)}-\textbf{(A3)}} hold and let $\ou$ be a
local solution of problem~\eqref{OptConPro}. Then, the perturbed
control problem~\eqref{PerturbPro} has at least one optimal control
$\ou_e$ with associated optimal perturbed state $y_{\ou_e+e_y}\in
H^1(\Omega)\cap C(\bar\Omega)$ for small $e$.
\end{Theorem}
\begin{proof}
Fix any parameter $e$ small enough. For each $u\in L^2(\Omega)$,
equation~\eqref{PerStateEq} has a unique weak solution
$y_{u+e_y}=G(u+e_y)\in H^1_0(\Omega)\cap C(\bar\Omega)$. Since
$\mathcal{U}^\varepsilon_{ad}$ is a bounded subset of
$L^\infty(\Omega)$ and thus it is bounded in $L^2(\Omega)$. Hence,
there is some constant $c>0$ such that
$$\|y_{u+e_y}\|_{H^1(\Omega)}+\|y_{u+e_y}\|_{C(\bar\Omega)}\leq c\|u+e_y\|_{L^2(\Omega)}.$$
The latter implies that there is some constant $M>0$ such that
$$\|y_{u+e_y}\|_{C(\bar\Omega)}\leq M,\ \forall u\in\mathcal{U}^\varepsilon_{ad}.$$
We observe that the functional $\mJ(u,e)$ is bounded from
below for $u\in\mathcal{U}^\varepsilon_{ad}$. Therefore, the infimum
$$\mJ_0(e)=\inf_{u\in\mathcal{U}^\varepsilon_{ad}}\mJ(u,e)$$
exists. Let $\{(y_n,u_n)\}_n$ is a minimizing sequence, that is, let
$u_n\in\mathcal{U}^\varepsilon_{ad}$ and $y_n=y_{u_n+e_y}$ be such
that $\mJ(u_n,e)\to\mJ_0(e)$ as $n\to\infty$. We
interpret $\mathcal{U}^\varepsilon_{ad}$ as a nonempty, closed,
bounded and convex set in the reflexive Banach space $L^2(\Omega)$,
thus $\mathcal{U}^\varepsilon_{ad}$ is weakly sequentially compact.
Hence, without loss of generality we may assume that $\{u_n\}$
converges weakly in $L^2(\Omega)$ to some $\ou_e$ in
$\mathcal{U}^\varepsilon_{ad}$, i.e.,
$$u_n\rightharpoonup\ou_e\quad\mbox{as}\ n\to\infty.$$
Consider the sequence
$$t_n(\cdot)=f(\cdot,y_n(\cdot)),\ \forall n\in\N.$$
Recall that, for every $n\in\N$, $\|y_n\|_{L^\infty(\Omega)}\leq M$
since $|y_n(x)|\leq M$ for all $x\in\bar\Omega$. Then $\{t_n\}$ is
also bounded in $L^\infty(\Omega)$, and thus bounded in
$L^2(\Omega)$. Without loss of generality we may assume that
$\{y_n\}$ and $\{t_n\}$ converges weakly in $L^2(\Omega)$ to some
$y\in L^2(\Omega)$ and $t\in L^2(\Omega)$ respectively.

Observe that $y_n$ solves the problem
$$\begin{cases}
\begin{aligned}
    Ay_n+y_n&=r_n\ &&\mbox{in}\ \Omega\\
         y_n&=0    &&\mbox{on}\ \Gamma,
\end{aligned}
\end{cases}$$
where $r_n:=-f(x,y_n)+y_n+u_n+e_y$ converges weakly in $L^2(\Omega)$
to $-t+y+\ou_e+e_y$. By virtue of \cite[Theorem~2.6]{Trolt10B}, the
mapping $r_n\mapsto y_n$ is linear and continuous from $L^2(\Omega)$
into $H^1(\Omega)$. Since every continuous linear operator is also
weakly continuous, $\{y_n\}$ must converge weakly in $H^1(\Omega)$
to some $y_e\in H^1(\Omega)$, i.e.,
$$y_n\rightharpoonup y_e\quad\mbox{as}\ n\to\infty.$$
Since $H^1(\Omega)$ is compactly embedded in $L^2(\Omega)$ by
\cite[Theorem~7.4]{Trolt10B}, we have
$$\|y_n-y_e\|_{L^2(\Omega)}\to0\quad\mbox{as}\ n\to\infty.$$
The set $\{v\in L^2(\Omega):\|v\|_{L^\infty(\Omega)}\leq M\}$ is
bounded, closed, and convex, thus it is also weakly sequentially
closed. Consequently, $y_e$ belongs to this set since $|y_n(x)|\leq
M$ for all $x\in\bar\Omega$.

We have
$$\int_{\Omega}\nabla y_n\cdot\nabla vdx+\int_{\Omega}f(\cdot,y_n)vdx=\int_{\Omega}(u_n+e_y)vdx,\
  \forall v\in H^1(\Omega).$$
Passing to the limit as $n\to\infty$, we see that
$y_n\rightharpoonup y_e$ in $H^1(\Omega)$ yields the convergence of
the first integral. In addition, $y_n\to y_e$ in $L^2(\Omega)$ and
$\|y_n\|_{L^\infty(\Omega)}\leq M$ yields the convergence of the
second integral. By dominated convergence theorem, we have
$$\int_{\Omega}\nabla y_e\cdot\nabla vdx+\int_{\Omega}f(\cdot,y_e)vdx=\int_{\Omega}(\ou_e+e_y)vdx,\
  \forall v\in H^1(\Omega),$$
i.e., $y_e\equiv y_{\ou_e+e_y}$ is the weak solution of
\eqref{PerStateEq} corresponding to the right-hand side $\ou_e+e_y$.

Since $u_n\rightharpoonup\ou_e$ in $L^2(\Omega)$, we have
$u_n+e_y\rightharpoonup\ou_e+e_y$ and thus $y_{u_n+e_y}\to
y_{\ou_e+e_y}$ in $C(\bar\Omega)$; see, e.g.,
\cite[Theorem~2.1]{CaReTr08SIOPT}. Consequently, we obtain
$$\begin{aligned}
    \mJ_0(e)
    &=\lim_{n\to\infty}\mJ(u_n,e)=\lim_{n\to\infty}\big(J(u_n+e_y)+(e_J,y_{u_n+e_y})_{L^2(\Omega)}\big)\\
    &=J(\ou_e+e_y)+(e_J,y_{\ou_e+e_y})_{L^2(\Omega)}=\mJ(\ou_e,e).
\end{aligned}$$
This concludes that $\ou_e$ is an optimal control of the perturbed
problem \eqref{PerturbPro} with associated optimal perturbed state
$y_{\ou_e+e_y}\in H^1(\Omega)\cap C(\bar\Omega)$. $\hfill\Box$
\end{proof}

\medskip
We recall that $y$ is an extremal point of a set $K$ if and only if
$y=\lambda y_1+(1-\lambda)y_2$ with $y_1,y_2\in K$ and $0<\lambda<1$
entails $y_1=y_2=y$. For  a generic set $S$ we denote its closed
convex hull by $\overline{\conv}S$.
\begin{Theorem}\label{ThmVisinL1}{\rm(See \cite[Theorem~1]{Visin84CPDE})}
Assume that $u_n\rightharpoonup u$ in $L^1(\Omega)$ and that $u(x)$
is an extremal point of
$K(x):=\overline{\conv}\big(\{u_n(x)\}_{n\in\N}\cup\{u(x\}\big)$ for
a.e. $x\in\Omega$. Then, $u_n\to u$ in $L^1(\Omega)$.
\end{Theorem}

The forthcoming theorem shows that under the assumptions
{\rm\textbf{(A1)}-\textbf{(A3)}}, the perturbed control problem
\eqref{PerturbPro} is stable in $L^2(\Omega)$.

\begin{Theorem}\label{ThmSolSbL2}
Assume that {\rm\textbf{(A1)}-\textbf{(A3)}} hold. Let $\ou$ be a
local bang-bang solution of problem~\eqref{OptConPro} in the
neighborhood $\oB^2_\varepsilon(\ou)$ and let $\ou_e$ be a global
solution of problem~\eqref{PerturbPro} with respect to the parameter
$e$. Then, we have $\ou_e\to\ou$ in $L^2(\Omega)$ as $e\to0$ in $E$.
\end{Theorem}
\begin{proof}
Let $\{e_n\}_{n\in\N}$ be such that $e_n\to0$ and let
$\ou_n=\ou_{e_n}\in\oB^2_\varepsilon(\ou)$ be a solution of
\eqref{PerturbPro}. Then, $\ou_n\rightharpoonup\wu$ in $L^2(\Omega)$
for some $\wu\in\oB^2_\varepsilon(\ou)$. Letting $n\to\infty$, we
have $\mJ(\ou_n,e_n)\leq\mJ(\ou,e_n)$, and
$$\mJ(\ou_n,e_n)\to\mJ(\wu,0)\quad\mbox{and}\quad\mJ(\ou,e_n)\to\mJ(\ou,0).$$
It follows that $\mJ(\wu,0)\leq\mJ(\ou,0)$. Hence,
we get $\wu=\ou$. We have shown that $\ou_n\rightharpoonup\ou$ in
$L^2(\Omega)$. This yields $\ou_n\rightharpoonup\ou$ in
$L^1(\Omega)$. We observe that $\ou(x)$ is an extremal point of the
set
$K(x):=\overline{\conv}\big(\{\ou_n(x)\}_{n\in\N}\cup\{\ou(x\}\big)$
for a.e. $x\in\Omega$ since $\ou$ is a bang-bang control. Applying
Theorem~\ref{ThmVisinL1} we infer that $\ou_n\to\ou$ in
$L^1(\Omega)$. Note that $\ou_n\in\mathcal{U}^\varepsilon_{ad}$ and
$\mathcal{U}^\varepsilon_{ad}$ is a bounded set in
$L^\infty(\Omega)$, thus $\|\ou_n-\ou\|_{L^\infty(\Omega)}$ is
bounded for every $n\in\N$. Therefore, there exists a constant $M>0$
such that
$$\|\ou_n-\ou\|_{L^2(\Omega)}\leq\|\ou_n-\ou\|^{1/2}_{L^1(\Omega)}\|\ou_n-\ou\|^{1/2}_{L^\infty(\Omega)}
  \leq M\|\ou_n-\ou\|^{1/2}_{L^1(\Omega)}\to0,$$
where the first inequality is due to H\"older's inequality.
$\hfill\Box$
\end{proof}

\begin{Theorem}\label{ThmHldrEstm}
Assume that {\rm\textbf{(A1)}-\textbf{(A3)}} hold and let $\ou$ be a
local bang-bang solution of problem~\eqref{OptConPro} under
assumption~{\rm\textbf{(A4.$\ae$)}} and condition~\eqref{SOrdCd} for
some $\delta>0$ and $\tau>0$. Then, the perturbed
problem~\eqref{PerProWtCd} has local solutions $\ou_e$ near $\ou$
for small $e$. Moreover, $\ou_e\to\ou$ in $L^2(\Omega)$ as $e\to0$
in $E$.
\end{Theorem}
\begin{proof}
According to Theorem~\ref{ThmSSC}, $\ou$ is a local bang-bang
solution of problem~\eqref{OptConPro} via \eqref{GrOpCd} which holds
for some $\varepsilon>0$. For any parameter $e\in E$, by
Theorem~\ref{ThmAxExSol} problem~\eqref{PerturbPro} has global
solutions $\ou_e\in\mathcal{U}_{ad}\cap\oB^2_\varepsilon(\ou)$ with
respect to the $\varepsilon$. By Theorem~\ref{ThmSolSbL2},
$\ou_e\to\ou$ in $L^2(\Omega)$ as $e\to0$ in $E$. In particular,
$\ou_e$ belongs to the interior of the ball $\oB^2_\varepsilon(\ou)$
for small $e$. This shows that $\ou_e$ is a local solution of the
perturbed problem~\eqref{PerProWtCd} near $\ou$ for small $e$.
$\hfill\Box$
\end{proof}

\subsection{Stability of KKT points of perturbed problems}

For the perturbed problem~\eqref{PerProWtCd}, we call a control
$u_e\in\mathcal{U}_{ad}$ \emph{the KKT point} of \eqref{PerProWtCd}
if $u_e$ satisfies the following system
\begin{equation}\label{PertbStateEq}
\begin{cases}
\begin{aligned}
    Ay_{u_e+e_y}+f(x,y_{u_e+e_y})&=u_e+e_y\ &&\mbox{in}\ \Omega\\
                      y_{u_e+e_y}&=0        &&\mbox{on}\ \Gamma,
\end{aligned}
\end{cases}
\end{equation}
\begin{equation}\label{PertbAdjEq}
\begin{cases}
\begin{aligned}
    A^*\varphi_{u_e,e}+\dfrac{\partial f}{\partial y}(x,y_{u_e+e_y})\varphi_{u_e,e}
                     &=\dfrac{\partial L}{\partial y}(x,y_{u_e+e_y})+e_J\ &&\mbox{in}\ \Omega\\
    \varphi_{u_e,e}&=0                                                    &&\mbox{on}\ \Gamma,
\end{aligned}
\end{cases}
\end{equation}
\begin{equation}\label{PertbVarIneq}
    \int_\Omega\varphi_{u_e,e}(x)\big(u(x)-u_e(x)\big)dx\geq0,\ \forall u\in\mathcal{U}_{ad}.
\end{equation}
In formula \eqref{PertbAdjEq}, we call $\varphi_{u_e,e}$ the adjoint
state of $y_{u_e+e_y}$ for perturbed
problem~\eqref{PerProWtCd}-\eqref{PerStaEqWtCd}. Then, the partial
derivative of $\mJ(u,e)$ in $u$ at $u_e$ can be computed by
the formula
$$\mJ'_u(u_e,e)(u-u_e)=\int_\Omega\varphi_{u_e,e}(x)\big(u(x)-u_e(x)\big)dx,\ \forall u\in\mathcal{U}_{ad}.$$
Let us denote $\varphi_{u_e+e_y}$ the adjoint state of $y_{u_e+e_y}$
for problem~\eqref{OptConPro}-\eqref{StateEq} with respect to the
control $u_e+e_y\in L^2(\Omega)$, i.e., $\varphi_{u_e+e_y}$ is the
weak solution of the following equation
\begin{equation}\label{PertbEyAdEq}
\begin{cases}
\begin{aligned}
    A^*\varphi_{u_e+e_y}+\dfrac{\partial f}{\partial y}(x,y_{u_e+e_y})\varphi_{u_e+e_y}
                     &=\dfrac{\partial L}{\partial y}(x,y_{u_e+e_y})\ &&\mbox{in}\ \Omega\\
    \varphi_{u_e+e_y}&=0                                              &&\mbox{on}\ \Gamma.
\end{aligned}
\end{cases}
\end{equation}

Observe that since $\varphi_{u_e,e}$ and $\varphi_{u_e+e_y}$ are
respectively the weak solutions of \eqref{PertbAdjEq} and
\eqref{PertbEyAdEq}, we deduce that
$\varphi_{u_e,e}-\varphi_{u_e+e_y}$ satisfies the following equation
$$\begin{cases}
  \begin{aligned}
     A^*(\varphi_{u_e,e}-\varphi_{u_e+e_y})
     +\dfrac{\partial f}{\partial y}(x,y_{u_e+e_y})(\varphi_{u_e,e}-\varphi_{u_e+e_y})
                                        &=e_J\ &&\mbox{in}\ \Omega\\
     \varphi_{u_e,e}-\varphi_{u_e+e_y}  &=0    &&\mbox{on}\ \Gamma.
\end{aligned}
\end{cases}$$
Therefore, there exists a constant $c_J>0$ such that
\begin{equation}\label{EstPSlVphi}
    \|\varphi_{u_e,e}-\varphi_{u_e+e_y}\|_{L^\infty(\Omega)}\leq c_J\|e_J\|_{L^2(\Omega)}.
\end{equation}

We now state some auxiliary results that will be used in the proofs
of the main results in this section.

\begin{Lemma}\label{Lem25CaEx}
Let there be given any $\ou\in\mathcal{U}_{ad}$. Then, there exists
a constant $C_1>0$ such that
\begin{equation}\label{Cas25Ext}
    \|y_u-y_{\ou}\|_Y+\|\varphi_u-\varphi_{\ou}\|_Y\leq C_1\|u-\ou\|_{L^2(\Omega)},\
    \forall u\in L^2(\Omega),
\end{equation}
where $y_u$ and $\varphi_u$ are respectively the weak solutions of
\eqref{StateEq} and \eqref{AdjStaEq}.
\end{Lemma}
\begin{proof}
For any $u\in L^2(\Omega)$, subtracting equations~\eqref{StateEq}
satisfied by $y_u$ and $y_{\ou}$ we obtain that
$$\begin{cases}
  \begin{aligned}
     A(y_u-y_{\ou})+f(x,y_u)-f(x,y_{\ou})&=u-\ou\ &&\mbox{in}\ \Omega\\
                              y_u-y_{\ou}&=0      &&\mbox{on}\ \Gamma.
\end{aligned}
\end{cases}$$
We see that
$$f(x,y_u(x))-f(x,y_{\ou}(x))
  =\int_0^1\frac{\partial f}{\partial y}\Big(x,y_u(x)+s\big(y_{\ou}(x)-y_u(x)\big)\Big)ds\big(y_u(x)-y_{\ou}(x)\big).$$
Due to assumption~\textbf{(A1)}, we have
$$\theta(x):=\int_0^1\frac{\partial f}{\partial y}\Big(x,y_u(x)+s\big(y_{\ou}(x)-y_u(x)\big)\Big)ds\geq0.$$
Let $\wy=y_u-y_{\ou}$ and $\wu=u-\ou$. Then, $\wy$ satisfies the
following equation
$$\begin{cases}
  \begin{aligned}
     A\wy+\theta(x)\wy&=\wu\ &&\mbox{in}\ \Omega\\
                   \wy&=0    &&\mbox{on}\ \Gamma,
\end{aligned}
\end{cases}$$
where $\theta(x)\geq0$. From this we can deduce that there exists a
constant $c_1>0$ such that
$$\|\wy\|_{H^1_0(\Omega)}+\|\wy\|_{L^\infty(\Omega)}\leq c_1\|\wu\|_{L^2(\Omega)},$$
or, equivalently, as follows
\begin{equation}\label{EstYuYou}
    \|y_u-y_{\ou}\|_Y\leq c_1\|u-\ou\|_{L^2(\Omega)}.
\end{equation}
Now, subtracting equations~\eqref{AdjStaEq} satisfied by $\varphi_u$
and $\varphi_{\ou}$ we obtain that
$$\begin{cases}
\begin{aligned}
    A^*(\varphi_u-\varphi_{\ou})+\dfrac{\partial f}{\partial y}(x,y_{\ou})(\varphi_u-\varphi_{\ou})
                           &=\dfrac{\partial L}{\partial y}(x,y_u)-\dfrac{\partial L}{\partial y}(x,y_{\ou})\\
                           &\quad-\bigg(\dfrac{\partial f}{\partial y}(x,y_u)
                                 -\dfrac{\partial f}{\partial y}(x,y_{\ou})\bigg)\varphi_u
                                                            \ &&\mbox{in}\ \Omega\\
    \varphi_u-\varphi_{\ou}&=0                                &&\mbox{on}\ \Gamma.
\end{aligned}
\end{cases}$$
This implies that $\|\varphi_u-\varphi_{\ou}\|_Y$ is estimated by
the $L^2(\Omega)$ norm of the right-hand side of the equation. Note
that $\|\varphi_u\|_Y\leq\ell_1$ for some constant $\ell_1>0$ and
$u\in L^2(\Omega)$. Using the assumptions \textbf{(A1)} and
\textbf{(A2)} and applying the mean value theorem we deduce that
$$\left|\dfrac{\partial L}{\partial y}(x,y_u)-\dfrac{\partial L}{\partial y}(x,y_{\ou})\right|
  +\left|\dfrac{\partial f}{\partial y}(x,y_u)-\dfrac{\partial f}{\partial y}(x,y_{\ou})\right|
  \leq(C_{L,M}+C_{f,M})\big|y_u(x)-y_{\ou}(x)\big|.$$
Arguing the same as in the proof of Theorem~\ref{ThmHldrEstm} we
obtain for some constant $D_{\alpha,\beta}>0$ that
$$\|y_u-y_{\ou}\|_{L^2(\Omega)}\leq D_{\alpha,\beta}\|u-\ou\|_{L^1(\Omega)}
  \leq|\Omega|^{1/2}D_{\alpha,\beta}\|u-\ou\|_{L^1(\Omega)}.$$
Consequently, we get the estimate
$$\begin{aligned}
    \left\|\dfrac{\partial L}{\partial y}(x,y_u)-\dfrac{\partial L}{\partial y}(x,y_{\ou})
           -\bigg(\dfrac{\partial f}{\partial y}(x,y_u)-\dfrac{\partial f}{\partial y}(x,y_{\ou})\bigg)\varphi_u
    \right\|_{L^2(\Omega)}
    \leq\ell_2\|u-\ou\|_{L^2(\Omega)},
\end{aligned}$$
where
$\ell_2:=(C_{L,M}+\ell_1C_{f,M})|\Omega|^{1/2}D_{\alpha,\beta}$.
Therefore, there exists a constant $c_2>0$ such that
\begin{equation}\label{EstVPuVPou}
    \|\varphi_u-\varphi_{\ou}\|_Y\leq c_2\|u-\ou\|_{L^2(\Omega)}.
\end{equation}
From \eqref{EstYuYou} and \eqref{EstVPuVPou} we have
$$\|y_u-y_{\ou}\|_Y+\|\varphi_u-\varphi_{\ou}\|_Y\leq C_1\|u-\ou\|_{L^2(\Omega)},$$
where $C_1=c_1+c_2$. $\hfill\Box$
\end{proof}

\begin{Theorem}\label{ThmStabKKT}
Assume that {\rm\textbf{(A1)}-\textbf{(A3)}} hold and let $\ou$ be a
local solution of problem~\eqref{OptConPro} under
assumption~{\rm\textbf{(A4.$\ae$)}} and condition~\eqref{SOrdCd} for
some $\delta>0$ and $\tau>0$. Then, there exist constants $\eta>0$
and $\varrho>0$ such that for any KKT point $u_e$ of
\eqref{PerProWtCd} with $\|u_e-\ou\|_{L^2(\Omega)}\leq\eta$, we have
$$\|u_e-\ou\|_{L^1(\Omega)}\leq\varrho\|e\|_E^\ae,$$
where $\ae$ is given in assumption~{\rm\textbf{(A4.$\ae$)}}.
\end{Theorem}
\begin{proof}
Let a parameter $e\in E$ be given and let $u_e$ be a KKT point of
\eqref{PerProWtCd}. We argue similarly to the proof of
Theorem~\ref{ThmSSC}. Due to Proposition \ref{PropFstCd} we get
$$\frac{\kappa}{2}\|u_e-\ou\|_{L^1(\Omega)}^{1+\frac{1}{\ae}}
  +\frac{1}{2}\int_\Omega|\varphi_{\ou}||u_e-\ou|dx
  \leq\big(J'(\ou)-\mJ'_u(u_e,e)\big)(u_e-\bar u).$$
By definition, we have
$$\mJ'_u(u_e,e)-J'(u_e)=\varphi_{u_e,e}-\varphi_{u_e}.$$
Recall the definitions of $\varphi_{u_e,e}$ and $\varphi_{u_e}$ in
\eqref{PertbAdjEq} and \eqref{AdjEq}. Due to \eqref{EstPSlVphi} and
\eqref{Cas25Ext}, we have
$$\begin{aligned}
    \|\varphi_{u_e,e}-\varphi_{u_e}\|_{L^\infty(\Omega)}
    &\leq\|\varphi_{u_e,e}-\varphi_{u_e+e_y}\|_{L^\infty(\Omega)}
     +\|\varphi_{u_e+e_y} - \varphi_{u_e}\|_{L^\infty(\Omega)}\\
    &\leq c(\|e_J\|_{L^2(\Omega)} + \|e_y\|_{L^2(\Omega)})\\
    &=c\|e\|_E
\end{aligned}$$
for some constant $c>0$. This proves
$$\frac{\kappa}{2}\|u_e-\ou\|_{L^1(\Omega)}^{1+\frac{1}{\ae}}+\frac{1}{2}\int_\Omega|\varphi_{\ou}||u_e-\ou|dx
  \leq\big(J'(\ou)-J'(u_e)\big)(u_e-\ou)+c\|e\|_E\|u_e-\bar u\|_{L^1(\Omega)}.$$
Making Taylor expansion, we find
$$\big(J'(u_e)-J'(\ou)\big)(u_e-\ou)=J''(\hu)(u_e-\bar u)^2$$
with $\hu=\ou+\theta(u_e-\ou)$ and $\theta\in (0,1)$. Introducing
the splitting $u_e-\bar u = v+w$ with $v\in C^\tau_{\bar u}$, see
again the proof of Theorem~\ref{ThmSSC}, we get
$$\big(J'(u_e)-J'(\ou)\big)(u_e-\ou)=J''(\ou)v^2+\big(J''(\hu)-J''(\ou)\big)v^2+J''(\hu)w^2+2J''(\hu)(v,w).$$
Proceeding exactly as in the proof of Theorem~\ref{ThmSSC}, we find
$$\big(J'(u_e)-J'(\ou)\big)(u_e-\ou)+\frac{1}{2}\int_\Omega|\varphi_{\ou}||u_e-\ou|dx
  \geq\frac{\delta}{8}\|z_{u_e-\ou}\|_{L^2(\Omega)}^2$$
for all $u_e$ with $\|u_e-\ou\|_{L^2(\Omega)}\leq\eta<\varepsilon$,
where $\varepsilon>0$ is as in the proof above and
$\eta\in(0,\varepsilon)$ is small enough. Putting everything
together, we obtain
$$\frac{\kappa}{2}\|u_e-\ou\|_{L^1(\Omega)}^{1+\frac{1}{\ae}}+\frac{\delta}{8}\|z_{u_e-\ou}\|_{L^2(\Omega)}^2
  \leq c\|e\|_E\|u_e-\ou\|_{L^1(\Omega)},$$
which shows
$$\|u_e-\ou\|_{L^1(\Omega)}\leq\varrho\|e\|_E^\ae,$$
where $\varrho:=2\kappa^{-1}c>0$. $\hfill\Box$
\end{proof}

\subsection{Stability of second-order conditions}

In this subsection, we will prove that the second-order sufficient
optimality condition \eqref{SOrdCd} is stable under small
perturbations. This stability result leads to a sufficient
optimality condition for the KKT points of the perturbed
problem~\eqref{PerProWtCd}.

\begin{Lemma}\label{Lem26CaEx}
Let us fix any $\ou\in\mathcal{U}_{ad}$. Then, there exist constants
$C_2>0$ and $C_3>0$ such that for every $u\in\mathcal{U}_{ad}$ and
$e_y\in L^2(\Omega)$ with $\|e_y\|_{L^2(\Omega)}$ small enough, we
have
$$\|z_{u+e_y,v}-z_v\|_Y\leq C_2\|u+e_y-\ou\|_{L^2(\Omega)}\|z_v\|_{L^2(\Omega)},\ \forall v\in L^2(\Omega),$$
$$\|z_{u+e_y,v}\|_{L^2(\Omega)}\leq C_3\|v\|_{L^1(\Omega)},\ \forall v\in L^1(\Omega),$$
where $z_{u+e_y,v}=G'(u+e_y)v$ and $z_v=G'(\ou)v$.
\end{Lemma}
\begin{proof}
Subtracting equations~\eqref{EqSolZuv} satisfied by $z_{u+e_y,v}$
and $z_v$ we obtain
$$A(z_{u+e_y,v}-z_v)+\frac{\partial f}{\partial y}(x,y_{u+e_y})(z_{u+e_y,v}-z_v)
  +\bigg(\frac{\partial f}{\partial y}(x,y_{u+e_y})-\frac{\partial f}{\partial y}(x,y_{\ou})\bigg)z_v=0,$$
or, equivalently, as follows
$$A(z_{u+e_y,v}-z_v)+\frac{\partial f}{\partial y}(x,y_{u+e_y})(z_{u+e_y,v}-z_v)
  +\frac{\partial^2f}{\partial y^2}(x,\hy)(y_{u+e_y}-y_{\ou})z_v=0,$$
where $\hy=y_{\ou}+\theta(y_{u+e_y}-y_{\ou})$ for some measurable
function $\theta(\cdot)$ with $0\leq\theta(x)\leq1$. Combining this
with \eqref{EstSolEqSt}, assumption~\textbf{(A1)}, and
\eqref{Cas25Ext} we deduce that
$$\begin{aligned}
    \|z_{u+e_y,v}-z_v\|_Y
    &\leq C_{f,M}\|y_{u+e_y,v}-y_{\ou}\|_{L^2(\Omega)}\|z_v\|_{L^2(\Omega)}\\
    &\leq C_{f,M}D_{\alpha,\beta}\|u+e_y-\ou\|_{L^1(\Omega)}\|z_v\|_{L^2(\Omega)}\\
    &\leq C_2\|u+e_y-\ou\|_{L^2(\Omega)}\|z_v\|_{L^2(\Omega)},
\end{aligned}$$
where $C_2=C_{f,M}D_{\alpha,\beta}|\Omega|^{1/2}$.

Arguing similarly as in the proof of \cite[Lemma~2.6]{Cas12SICON} we
can use a regularity result for equation~\eqref{EqSolZuv} to obtain
the inequality
\begin{equation}\label{EstW1p0L1}
    \|z_{u+e_y,v}\|_{W^{1,p}_0(\Omega)}\leq C_p\|v\|_{L^1(\Omega)},
\end{equation}
where $C_p$ is independent of $u+e_y$ since $\mathcal{U}_{ad}$ is
bounded in $L^\infty(\Omega)$. Moreover, we can take $p$ close
enough to $N/(N-1)$ to have the embedding
$W^{1,p}_0(\Omega)\hookrightarrow L^2(\Omega)$. From this and
\eqref{EstW1p0L1} it follows that $\|z_{u+e_y,v}\|_{L^2(\Omega)}\leq
C_3\|v\|_{L^1(\Omega)}$ for some constant $C_3>0$. $\hfill\Box$
\end{proof}

\begin{Lemma}\label{Lem27CaEx}
Let us fix any $\ou\in\mathcal{U}_{ad}$. For every $\varepsilon>0$,
there exists $\rho>0$ such that for any $u\in\mathcal{U}_{ad}$ and
$e_y\in L^2(\Omega)$ with $\|e_y\|_{L^2(\Omega)}$ small enough and
$\|u+e_y-\ou\|_{L^2(\Omega)}\leq\rho$ the following inequality holds
\begin{equation}\label{Cas27Ext}
    \big|J''(u+e_y)v^2-J''(\ou)v^2\big|\leq\varepsilon\|z_v\|^2_{L^2(\Omega)},\ \forall v\in L^2(\Omega).
\end{equation}
\end{Lemma}
\begin{proof}
Let us define the function $F:\Omega\times\mathcal{U}_{ad}\to
L^\infty(\Omega)$ by setting
$$F(x,u)=\frac{\partial^2L}{\partial y^2}(x,y_u)-\varphi_u\frac{\partial^2f}{\partial y^2}(x,y_u),$$
where $y_u$ and $\varphi_u$ are respectively the weak solutions of
\eqref{StateEq} and \eqref{AdjStaEq}. Using the assumptions
\textbf{(A1)} and \textbf{(A2)}, and \eqref{Cas25Ext} we deduce that
for any $\varepsilon>0$ there exists $\rho_1\in(0,1)$ such that
$$\|F(x,u+e_y)-F(x,\ou)\|_{L^\infty(\Omega)}\leq\frac{\varepsilon}{2}\quad\mbox{whenever}\quad
  \|u+e_y-\ou\|_{L^2(\Omega)}\leq\rho_1.$$
We also have
$$\|F(x,u+e_y)\|_\infty\leq K_M,\ \forall u\in\mathcal{U}_{ad}.$$
Using the above estimates for $F$ together with
Lemma~\ref{Lem26CaEx} and arguing similarly as in the proof of
\cite[Lemma~2.7]{Cas12SICON} we obtain \eqref{Cas27Ext}.
$\hfill\Box$
\end{proof}

\medskip
According to \cite[Theorem~2.4]{Cas12SICON}, under assumptions
{\rm\textbf{(A1)}-\textbf{(A3)}}, if a feasible control $\ou$ of
problem~\eqref{OptConPro} satisfies condition~\eqref{SOrdCd} for
some $\delta>0$ and $\tau>0$, then it is a local solution of
\eqref{OptConPro}. The following theorem shows that
condition~\eqref{SOrdCd} is stable for small parameter $e$.

\begin{Theorem}\label{ThmPerbSSC}
Assume that {\rm\textbf{(A1)}-\textbf{(A3)}} hold and let
$\ou\in\mathcal{U}_{ad}$ be such that condition~\eqref{SOrdCd} holds
at $\ou$ for some $\delta>0$ and $\tau>0$. Then, there exist
$\widehat{\delta}>0$ and $\rho_1>0$ such that for $e\in E$ small
enough and $\|u+e_y-\ou\|_{L^2(\Omega)}\leq\rho_1$ we have
\begin{equation}\label{PerbSOrdCd}
    \mJ''_u(u,e)v^2\geq\widehat{\delta}\|z^e_{u,v}\|^2_{L^2(\Omega)},\
    \forall v\in C^\tau_{\ou},
\end{equation}
where $z^e_{u,v}=G'(u+e_y)v$ is the solution of \eqref{EqSolZuv} for
$y=y_{u+e_y}$.
\end{Theorem}
\begin{proof}
Recall that
$\mJ(u,e)=J(u+e_y)+(e_J,y_{u+e_y})_{L^2(\Omega)}$ with
$e=(e_J,e_y)\in E$. It follows that
$$\mJ''_u(u,e)v^2=J''(u+e_y)v^2+\big(e_J,G''(u+e_y)v^2\big)_{L^2(\Omega)},\
  \forall v\in L^2(\Omega).$$
By Lemma~\ref{Lem27CaEx}, for any $\varepsilon_1>0$, there exists
$\rho_1>0$ such that if $\|u+e_y-\ou\|_{L^2(\Omega)}\leq\rho_1$, we
have
$$\big|J''(u+e_y)v^2-J''(\ou)v^2\big|\leq\varepsilon_1\|z_v\|^2_{L^2(\Omega)},\ \forall v\in L^2(\Omega).$$
We can choose $\varepsilon_1<\delta$ and choose small $e$ satisfying
$\|u+e_y-\ou\|_{L^2(\Omega)}\leq\rho_1$. Then, it holds that
$$\begin{aligned}
    J''(u+e_y)v^2
    &\geq J''(\ou)v^2-\varepsilon_1\|z_v\|^2_{L^2(\Omega)}\\
    &\geq\delta\|z_v\|^2_{L^2(\Omega)}-\varepsilon_1\|z_v\|^2_{L^2(\Omega)}
     =(\delta-\varepsilon_1)\|z_v\|^2_{L^2(\Omega)}.
\end{aligned}$$
Hence, for every $v\in L^2(\Omega)$, we have
\begin{equation}\label{PerbSeDeJ}
    \mJ''_u(u,e)v^2
    \geq(\delta-\varepsilon_1)\|z_v\|^2_{L^2(\Omega)}+\big(e_J,G''(u+e_y)v^2\big)_{L^2(\Omega)}.
\end{equation}
Note that $G''(u+e_y)v^2$ is a weak solution of \eqref{EqSolSeGvv}
satisfying the condition for some constant $C\geq0$ as follows
\begin{equation}\label{PerbEsSeG}
\begin{aligned}
    \|G''(u+e_y)v^2\|_{L^2(\Omega)}
    &\leq C\left\|-\frac{\partial^2f}{\partial y^2}(x,y_{u+e_y})z_{u+e_y,v}^2\right\|_{L^2(\Omega)}\\
    &\leq C\left\|\frac{\partial^2f}{\partial y^2}(x,y_{u+e_y})\right\|_{L^\infty(\Omega)}
          \|z_{u+e_y,v}\|_{L^2(\Omega)}^2
     \leq CC_{f,M}\|z^e_{u,v}\|_{L^2(\Omega)}^2,
\end{aligned}
\end{equation}
where $z^e_{u,v}=z_{u+e_y,v}=G'(u+e_y)v$. In addition, by
Lemma~\ref{Lem26CaEx}, we have
$$\|z^e_{u,v}-z_v\|_{L^2(\Omega)}\leq C_2\|u+e_y-\ou\|_{L^2(\Omega)}\|z_v\|_{L^2(\Omega)}.$$
Hence, for $e$ small enough, we deduce that
\begin{equation}\label{EsZuev}
\begin{aligned}
    \|z^e_{u,v}\|_{L^2(\Omega)}
    &\leq\|z_v\|_{L^2(\Omega)}+C_2\|u+e_y-\ou\|_{L^2(\Omega)}\|z_v\|_{L^2(\Omega)}\\
    &\leq\big(1+C_2\|u+e_y-\ou\|_{L^2(\Omega)}\big)\|z_v\|_{L^2(\Omega)}\\
    &\leq(1+C_2\rho_1)\|z_v\|_{L^2(\Omega)}.
\end{aligned}
\end{equation}
Combining this with \eqref{PerbSeDeJ}, \eqref{PerbEsSeG}, and
\eqref{EsZuev} we get
$$\begin{aligned}
    \mJ''_u(u,e)v^2
    &\geq(\delta-\varepsilon_1)\|z_v\|^2_{L^2(\Omega)}-\|e_J\|_{L^2(\Omega)}\|G''(u+e_y)v^2\|_{L^2(\Omega)}\\
    &\geq(\delta-\varepsilon_1)(1+C_2\rho_1)^{-1}\|z^e_{u,v}\|^2_{L^2(\Omega)}
     -\|e_J\|_{L^2(\Omega)}CC_{f,M}\|z^e_{u,v}\|_{L^2(\Omega)}^2\\
    &=\big((\delta-\varepsilon_1)(1+C_2\rho_1)^{-1}-\|e_J\|_{L^2(\Omega)}CC_{f,M}\big)\|z^e_{u,v}\|_{L^2(\Omega)}^2.
\end{aligned}$$
We have shown that
\begin{equation}\label{TemEsSeJ}
    \mJ''_u(u,e)v^2\geq\widehat{\delta}\|z^e_{u,v}\|_{L^2(\Omega)}^2,\ \forall v\in C^\tau_{\ou},
\end{equation}
where
$\widehat{\delta}:=(\delta-\varepsilon_1)(1+C_2\rho_1)^{-1}-\|e_J\|_{L^2(\Omega)}CC_{f,M}>0$
for $e$ small enough. $\hfill\Box$
\end{proof}

\medskip
Let us mention that it is an open problem to prove that a
second-order condition is fulfilled with respect to the critical
cone to $u_e$, i.e., replace $C^\tau_{\bar u}$ by $C^\tau_{u_e}$ in
\eqref{SOrdCdKKT}. If one assumes such a condition, then KKT points
of the perturbed problem are indeed local solutions.

\begin{Theorem}
Assume that {\rm\textbf{(A1)}-\textbf{(A3)}} hold and that $u_e$ is
a KKT point of problem~\eqref{PerProWtCd} such that there exist
$\delta>0$ and $\tau>0$ such that
\begin{equation}\label{SOrdCdKKT}
    J''(u_e)v^2\geq\delta\|z_{u_e,v}\|^2_{L^2(\Omega)},\ \forall v\in C^\tau_{u_e},
\end{equation}
where $z_{u_e,v}=G'(u_e)v$ is the solution of \eqref{EqSolZuv} for
$y=y_{u_e}$. Then, there exist $\eta>0$ and $\widehat{\delta}>0$
such that
\begin{equation}\label{GrSffCod}
\begin{aligned}
    \mJ(u_e,e)+\frac{3\widehat{\delta}}{16}\|z^e_{\wu,u-u_e}\|^2_{L^2(\Omega)}\leq\mJ(u,e),~
    \forall u\in\mathcal{U}_{ad}\cap\oB^2_\eta(u_e),
\end{aligned}
\end{equation}
where $\wu=u_e+\theta(u-u_e)$ for some $\theta\in(0,1)$ and
$z^e_{\wu,u-u_e}=G'(\wu+e_y)(u-u_e)$.
\end{Theorem}
\begin{proof}
Let us define the function $F_e:\Omega\times\mathcal{U}_{ad}\to
L^\infty(\Omega)$ by setting
$$F_e(x,u)=\frac{\partial^2L}{\partial y^2}(x,y_{u+e_y})
  -\varphi_{u+e_y}\frac{\partial^2f}{\partial y^2}(x,y_{u+e_y}).$$
Then, for every $e$ small enough, $F_e$ is well-defined due to the
assumptions on $f$ and $L$, and
$$\|y_{u+e_y}\|_{L^\infty(\Omega)}+\|\varphi_{u+e_y}\|_{L^\infty(\Omega)}\leq M,\ \forall u\in\mathcal{U}_{ad},$$
for some $M>0$. We also have
$$\|F_e(x,u)\|_\infty\leq K_M,\ \forall u\in\mathcal{U}_{ad}.$$
For every $u\in\mathcal{U}_{ad}$ and $v,w\in L^2(\Omega)$, it holds
that
$$\begin{aligned}
    |\mJ''_u(u,e)(v,w)|
    &\leq|J''(u+e_y)(v,w)|+\big|\big(e_J,G''(u+e_y)(v,w)\big)_{L^2(\Omega)}\big|\\
    &=\left|\int_\Omega F\big(x,(u+e_y)(x)\big)z_{u+e_y,v}z_{u+e_y,w}dx\right|
      +\|e_J\|_{L^2(\Omega)}\|G''(u+e_y)(v,w)\|_{L^2(\Omega)}\\
    &\leq\left|\int_\Omega F_e(x,u(x))z^e_{u,v}z^e_{u,w}dx\right|
      +\|e_J\|_{L^2(\Omega)}CC_{f,M}\|z^e_{u,v}\|_{L^2(\Omega)}\|z^e_{u,w}\|_{L^2(\Omega)}\\
    &\leq\big(K_M+\|e_J\|_{L^2(\Omega)}CC_{f,M}\big)\|z^e_{u,v}\|_{L^2(\Omega)}\|z^e_{u,w}\|_{L^2(\Omega)}.
\end{aligned}$$
By setting $K_{M,e}:=K_M+\|e_J\|_{L^2(\Omega)}CC_{f,M}$, we obtain
$$|\mJ''_u(u,e)(v,w)|\leq K_{M,e}\|z^e_{u,v}\|_{L^2(\Omega)}\|z^e_{u,w}\|_{L^2(\Omega)},\
  \forall u\in\mathcal{U}_{ad},\forall v,w\in L^2(\Omega).$$

Using the assumptions of the theorem and applying
Theorem~\ref{ThmPerbSSC} we can find $\widehat{\delta}>0$ and
$\rho_1>0$ such that when $e\in E$ is small enough and
$\|u+e_y-u_e\|_{L^2(\Omega)}\leq\rho_1$, we have
\begin{equation}\label{PerbSOCdUe}
    \mJ''_u(u,e)v^2\geq\widehat{\delta}\|z^e_{u,v}\|^2_{L^2(\Omega)},\
    \forall v\in C^\tau_{u_e}.
\end{equation}
Let us fix $\eta$ with $0<\eta<\rho_1$ such that one can find a
constant $\widehat{\tau}>0$ satisfying the condition for all
$u\in\mathcal{U}_{ad}\cap\oB^2_\eta(u_e)$ that
\begin{equation}\label{DKeta}
    \tau-c_J\|e_J\|_{L^2(\Omega)}-C_1\|e_y\|_{L^2(\Omega)}\geq\widehat{\tau}>0,\ \mbox{for all small}\ \|e\|_E,
\end{equation}
and
\begin{equation}\label{DKwdelta}
    \frac{\widehat{\tau}}{C_3^2\eta\sqrt{|\Omega|}}
    -\frac{K_{M,e}}{2}-\frac{2K^2_{M,e}}{\widehat{\delta}}\geq\frac{3\widehat{\delta}}{8},
\end{equation}
where $C_1$, $c_J$, and $\widehat{\delta}$ are respectively given in
Lemma~\ref{Lem25CaEx}, \eqref{EstPSlVphi}, and
Theorem~\ref{ThmPerbSSC}.

When $u\in\mathcal{U}_{ad}\cap\oB^2_\eta(u_e)$, we define
$$v(x)=\begin{cases}
          u(x)-u_e(x),      &\mbox{if}\ |\varphi_{u_e}(x)|\leq\tau\\
          0,                &\mbox{otherwise}
       \end{cases}
  \qquad\mbox{and}\quad w=(u-u_e)-v.$$
We see that $v\in C^\tau_{u_e}$. Using a Taylor expansion of second
order for $\mJ(\cdot,e)$, we have
$$\begin{aligned}
    \mJ(u,e)
    &=J(u+e_y)+(e_J,y_{u+e_y})_{L^2(\Omega)}\\
    &=J(u_e+e_y)+(e_J,y_{u_e+e_y})_{L^2(\Omega)}\\
    &\quad+J'(u_e+e_y)(u-u_e)+\big(e_J,G'(u_e+e_y)(u-u_e)\big)_{L^2(\Omega)}\\
    &\quad+\frac{1}{2}J''(\hu)(u-u_e)^2+\frac{1}{2}\big(e_J,G''(\hu)(u-u_e)^2\big)_{L^2(\Omega)}\\
    &=\mJ(u_e,e)+\mJ'_u(u_e,e)(u-u_e)
      +\frac{1}{2}J''(\hu)(u-u_e)^2+\frac{1}{2}\big(e_J,G''(\hu)(u-u_e)^2\big)_{L^2(\Omega)}\\
    &=\mJ(u_e,e)+\int_\Omega\varphi_{u_e,e}(u-u_e)dx
      +\frac{1}{2}J''(\hu)(v+w)^2+\frac{1}{2}\big(e_J,G''(\hu)(v+w)^2\big)_{L^2(\Omega)},
\end{aligned}$$
where $u-u_e=v+w$ and $\hu=u_e+e_y+\theta(u-u_e)$ for some
$\theta\in(0,1)$. By \eqref{PertbVarIneq}, we have
$$\begin{aligned}
    \mJ(u,e)
    &=\mJ(u_e,e)+\int_\Omega|\varphi_{u_e,e}||v+w|dx
     +\frac{1}{2}\Big(J''(\hu)v^2+\big(e_J,G''(\hu)v^2\big)_{L^2(\Omega)}\Big)\\
    &\quad+\frac{1}{2}\Big(J''(\hu)w^2+\big(e_J,G''(\hu)w^2\big)_{L^2(\Omega)}\Big)
     +J''(\hu)(v,w)+\big(e_J,G''(\hu)(v,w)\big)_{L^2(\Omega)}\\
    &\geq\mJ(u_e,e)+\int_\Omega|\varphi_{u_e,e}||w|dx
     +\frac{1}{2}\mJ''_u(\wu,e)v^2
     +\frac{1}{2}\mJ''_u(\wu,e)w^2+\mJ''_u(\wu,e)(v,w),
\end{aligned}$$
where $\wu=\hu-e_y=u_e+\theta(u-u_e)\in\mathcal{U}_{ad}$.

Note that $\|\wu+e_y-u_e\|\leq\rho_1$ for $e$ small enough. Thus, by
\eqref{PerbSOCdUe}, we have
$$\mJ''_u(\wu,e)v^2\geq\widehat{\delta}\|z^e_{\wu,v}\|^2_{L^2(\Omega)}.$$
Now, from \eqref{EstPSlVphi} and \eqref{Cas25Ext} we deduce that
$$\begin{aligned}
    \|\varphi_{u_e,e}-\varphi_{u_e}\|_{L^\infty(\Omega)}
    &\leq\|\varphi_{u_e,e}-\varphi_{u_e+e_y}\|_{L^\infty(\Omega)}
         +\|\varphi_{u_e+e_y}-\varphi_{u_e}\|_{L^\infty(\Omega)}\\
    &\leq c_J\|e_J\|_{L^2(\Omega)}+C_1\|e_y\|_{L^2(\Omega)}.
\end{aligned}$$
This implies that
$$|\varphi_{u_e}(x)|-|\varphi_{u_e,e}(x)|\leq\|\varphi_{u_e,e}-\varphi_{u_e}\|_{L^\infty(\Omega)}
  \leq c_J\|e_J\|_{L^2(\Omega)}+C_1\|e_y\|_{L^2(\Omega)},$$
and thus
$$|\varphi_{u_e,e}(x)|\geq|\varphi_{u_e}(x)|-c_J\|e_J\|_{L^2(\Omega)}-C_1\|e_y\|_{L^2(\Omega)}.$$
From this and \eqref{DKeta} we obtain
$$\begin{aligned}
    \int_\Omega|\varphi_{u_e,e}||w|dx
    &\geq\int_\Omega|\varphi_{u_e}||w|dx
     -\int_\Omega\big(c_J\|e_J\|_{L^2(\Omega)}+C_1\|e_y\|_{L^2(\Omega)}\big)|w|dx\\
    &\geq\tau\|w\|_{L^1(\Omega)}
     -\big(c_J\|e_J\|_{L^2(\Omega)}+C_1\|e_y\|_{L^2(\Omega)}\big)\|w\|_{L^1(\Omega)}\\
    &=\big(\tau-c_J\|e_J\|_{L^2(\Omega)}-C_1\|e_y\|_{L^2(\Omega)}\big)\|w\|_{L^1(\Omega)}\\
    &\geq\widehat{\tau}\|w\|_{L^1(\Omega)}.
\end{aligned}$$
By Schwarz's inequality we get
$$\|w\|_{L^1(\Omega)}\leq\|w\|_{L^2(\Omega)}\sqrt{|\Omega|}
  \leq\|u-u_e\|_{L^2(\Omega)}\sqrt{|\Omega|}\leq\eta\sqrt{|\Omega|}.$$
Consequently, by Lemma~\ref{Lem26CaEx} we have
$$\frac{\widehat{\tau}}{C_3^2\eta\sqrt{|\Omega|}}\|z^e_{\wu,w}\|^2_{L^2(\Omega)}
  \leq\frac{\widehat{\tau}}{\eta\sqrt{|\Omega|}}\|w\|^2_{L^1(\Omega)}\leq\widehat{\tau}\|w\|_{L^1(\Omega)}.$$

Summarizing the above estimates and using Young's inequality, we
deduce that
$$\begin{aligned}
    \mJ(u,e)
    &\geq\mJ(u_e,e)+\widehat{\tau}\|w\|_{L^1(\Omega)}
     +\frac{\widehat{\delta}}{2}\|z^e_{\wu,v}\|^2_{L^2(\Omega)}\\
    &\quad-\frac{K_{M,e}}{2}\|z^e_{\wu,w}\|^2_{L^2(\Omega)}
     -K_{M,e}\|z^e_{\wu,v}\|_{L^2(\Omega)}\|z^e_{\wu,w}\|_{L^2(\Omega)}\\
    &\geq\mJ(u_e,e)+\frac{\widehat{\tau}}{C_3^2\eta\sqrt{|\Omega|}}\|z^e_{\wu,w}\|^2_{L^2(\Omega)}
     +\frac{\widehat{\delta}}{2}\|z^e_{\wu,v}\|^2_{L^2(\Omega)}\\
    &\quad-\frac{K_{M,e}}{2}\|z^e_{\wu,w}\|^2_{L^2(\Omega)}
     -\frac{\widehat{\delta}}{8}\|z^e_{\wu,v}\|^2_{L^2(\Omega)}
     -\frac{2K^2_{M,e}}{\widehat{\delta}}\|z^e_{\wu,w}\|^2_{L^2(\Omega)}\\
    &=\mJ(u_e,e)+\frac{3\widehat{\delta}}{8}\|z^e_{\wu,v}\|^2_{L^2(\Omega)}
     +\bigg(\frac{\widehat{\tau}}{C_3^2\eta\sqrt{|\Omega|}}
     -\frac{K_{M,e}}{2}-\frac{2K^2_{M,e}}{\widehat{\delta}}\bigg)\|z^e_{\wu,w}\|^2_{L^2(\Omega)}.
\end{aligned}$$
By the choice of small $\eta>0$ satisfying
condition~\eqref{DKwdelta}, we obtain
\begin{equation}\label{KKTLocOpt}
\begin{aligned}
    \mJ(u,e)
    &\geq\mJ(u_e,e)+\frac{3\widehat{\delta}}{8}\|z^e_{\wu,v}\|^2_{L^2(\Omega)}
     +\frac{3\widehat{\delta}}{8}\|z^e_{\wu,w}\|^2_{L^2(\Omega)}\\
    &\geq\mJ(u_e,e)+\frac{3\widehat{\delta}}{16}\|z^e_{\wu,v}+z^e_{\wu,w}\|^2_{L^2(\Omega)}\\
    &=\mJ(u_e,e)+\frac{3\widehat{\delta}}{16}\|z^e_{\wu,u-u_e}\|^2_{L^2(\Omega)},
\end{aligned}
\end{equation}
which yields \eqref{GrSffCod}. $\hfill\Box$
\end{proof}

\subsection*{Conclusion}

We have studied perturbed bang-bang controls problems and obtained
H\"older stability results in $L^1(\Omega)$.
In this paper, the perturbations act linear in the state and adjoint equations.
Based on this work, one can discuss nonlinear perturbations without any additional difficulties.
Let us mention, that it is an open problem to prove necessary results for the H\"older stability of
bang-bang controls as for instance the results of \cite{MorNgia14SIOPT} do not apply directly.


\begin{thebibliography}{99}
\bibitem{BonSha00B}{\sc J.~F. Bonnans, A. Shapiro}.
{\em Perturbation Analysis of Optimization Problems},
Springer-Verlag, New York, 2000.

\bibitem{Cas12SICON}{\sc E. Casas}.
{\em Second order analysis for bang-bang control problems of PDEs},
SIAM J. Control Optim., 50 (2012), pp.~2355--2372.

\bibitem{CaReTr08SIOPT}{\sc E. Casas, J.~C. de los Reyes, F. Tr\"{o}ltzsch}.
{\em Sufficient second-order optimality conditions for semilinear
control problems with pointwise state constraints}, SIAM J. Optim.,
19 (2008), pp.~616--643.

\bibitem{CasDWchGWch17}{\sc E. Casas, D. Wachsmuth, G. Wachsmuth}.
{\em Sufficient second-order conditions for bang-bang control
problems}, Preprint, (2017), pp.~1--25.

\bibitem{Fel03SICON}{\sc U. Felgenhauer}.
{\em On stability of bang-bang type controls}, SIAM J. Control
Optim., 41 (2003), pp.~1843--1867.

\bibitem{Fel09CC}{\sc U. Felgenhauer, L. Poggiolini, G. Stefani}.
{\em Optimality and stability result for bang-bang optimal controls
with simple and double switch behaviour}, Control Cybernet. 38
(2009), pp.~1305--1325.

\bibitem{MalaTrol00CC}{\sc K. Malanowski, Tr\"{o}ltzsch}.
{\em Lipschitz stability of solutions to parametric optimal control
for elliptic equations}, Control Cybernet. 29 (2000), pp.~237--256.

\bibitem{MePaSch11NFAO}{\sc C. Meyer, L. Panizzi, A. Schiela}.
{\em Uniqueness criteria for the adjoint equation in
state-constrained elliptic optimal control}, Numer. Funct. Anal.
Optim., 32 (2011), pp.~983--1007.

\bibitem{MorNgia14SIOPT}{\sc B.~S. Mordukhovich, T.~T.~A. Nghia}.
{\em Full Lipschitzian and H\"{o}lderian stability in optimization
with applications to mathematical programming and optimal control},
SIAM J. Optim., 24 (2014), pp.~1344--1381.

\bibitem{PonWch16OPTIM}{\sc F. P\"{o}rner, D. Wachsmuth}.
{\em An iterative Bregman regularization method for optimal control
problems with inequality constraints}, Optimization, 65 (2016),
pp.~2195--2215.

\bibitem{PonWch17}{\sc F. P\"{o}rner, D. Wachsmuth}.
{\em Tikhonov regularization of optimal control problems governed by
semi-linear partial differential equations}, Preprint, (2017),
pp.~1--25.

\bibitem{Trolt10B}{\sc F. Tr\"{o}ltzsch}.
{\em Optimal Control of Partial Differential Equations. Theory,
Methods and Applications}, American Mathematical Society,
Providence, RI, 2010.

\bibitem{Visin84CPDE}{\sc A. Visintin}.
{\em Strong convergence results related to strict convexity}, Comm.
Partial Differential Equations, 9 (1984), pp.~439--466.
\end{thebibliography}
\end{document}